\documentclass{amsart}
\usepackage{amssymb}
\usepackage{amsfonts}
\usepackage{stmaryrd}
\usepackage{mathrsfs}
\usepackage{graphicx}

\begin{document}

\newtheorem{thm}{Theorem}[section]
\newtheorem{lem}[thm]{Lemma}
\newtheorem{cor}[thm]{Corollary}
\newtheorem{pro}[thm]{Proposition}
\newtheorem{exm}[thm]{Example}

\theoremstyle{definition}
\newtheorem{defn}{Definition}[section]
\newtheorem{clm}{Claim}[section]
\newtheorem{fact}[clm]{Fact}
\theoremstyle{remark}
\newtheorem{rmk}{Remark}[section]
\newtheorem{note}[rmk]{Note}
\newtheorem{ques}[rmk]{Question}

\def\square{\hfill${\vcenter{\vbox{\hrule height.4pt \hbox{\vrule
width.4pt height7pt \kern7pt \vrule width.4pt} \hrule height.4pt}}}$}
\def\T{\mathcal T}

\newenvironment{pf}{{\it Proof:}\quad}{\square \vskip 12pt}

\title{ 3-manifolds admitting locally large distance 2 Heegaard splittings}
\author{  Ruifeng Qiu and Yanqing Zou }

\thanks{This work was partially supported by  NSFC (No. 11271058, 11571110).}


\begin{abstract}
From the view of Heegaard splitting, it is known that if a closed orientable 3-manifold admits a distance at least three Heegaard splitting, then it is hyperbolic. However, for a closed orientable 3-manifold admitting only distance at most two Heegaard splittings, there are examples showing that it could be reducible, Seifert,  toroidal or hyperbolic. According to Thurston's Geometrization conjecture, the most important piece in eight geometries is hyperbolic. So for  a 3-manifold admitting a distance two Heegaard splittings, it is critical to determine the hyperbolicity of it in studying Heegaard splittings.

Inspired by the construction of hyperbolic 3-manifolds admitting  distance two Heegaard splittings in [Qiu, Zou and Guo,
Pacific J. Math. 275 (2015), no. 1, 231-255],  we introduce the definition of a locally large geodesic in curve complex and also a locally large distance two Heegaard splitting. Then we prove that if a 3-manifold admits a locally large distance two Heegaard splitting,
then it is a hyperbolic manifold or an amalgamation of a hyperbolic manifold and a Seifert manifold along
an incompressible torus, while the example in Section \ref{sec3} shows that there is a non hyperbolic 3-manifold in this case. After examining
those non hyperbolic cases, we give a sufficient and necessary condition for a hyperbolic 3-manifold when it admits a locally large distance two Heegaard splitting.

\end{abstract}

\maketitle

\vspace*{0.5cm} {\bf Keywords}: Hyperbolic 3-Manifold, Heegaard Distance, Curve Complex, Locally Large Geodesic.\vspace*{0.5cm}



\section{Introduction}
In 1898, Heegaard \cite{Hee} introduced a Heegaard splitting for a closed, orientable, triangulated 3-manifold, i.e., there is a closed, orientable
surface cutting this manifold into two handlebodies. Later, Moise \cite{MO} proved that every closed, orientable 3-manifold admits a triangulation.  So each closed orientable 3-manifold admits a Heegaard splitting. This makes studying 3-manifolds through Heegaard splittings possible.

One astonishing result proved by Haken \cite{Ha} is that if all Heegaard splitting of a 3-manifold  are reducible, i.e., there is  an essential simple closed curve in Heegaard surface bounding essential disks on both sides, then this manifold is reducible. Later, Casson and Gordon \cite{cg} defined  a weakly reducible Heegaard splitting and proved that if a 3-manifold has a weakly reducible and irreducible Heegaard splitting,  then it contains an embedded closed incompressible surface, i.e., it is Haken. Both of these two phenomenons drive people to think how Heegaard splittings reflect  3-manifolds.

For classifying 3-manifolds, Thurston \cite{thurston} introduced the Geometrization conjecture (Haken version proved by Thurston \cite{thurston} and full version proved by Perelman \cite{Perelman01, Perelman02, Perelman03}) as follows: for any closed, irreducible, orientable 3-manifold, there are finitely many disjoint, non isotopy essential tori so that after cutting the manifold along those tori, each piece is one of eight geometries. Among all of these eight geometries, one  is hyperbolic,  another one is solvable and the left six pieces are Seifert. In these eight geometries, it is known that Seifert 3-manifolds have been completely classified. Moreover all of their irreducible Heegaard splittings are either vertical or horizontal, see \cite{MS}. Cooper and Scharlemann \cite{CS} studied all irreducible Heegaard splittings of a solvmanifold. And there are series of works on Heegaard splittings of some typical 3-manifolds, such as Lens space, surface $\times S^{1}$ etc.

With the curve complex defined by Harvey \cite{h81},  Hempel \cite{h01} introduced an index-Heegaard distance for studying Heegaard splitting. Basically, this index- Heegaard distance is defined to be the length of a shortest geodesic in curve complex which connects these two boundaries of essential disks from different sides. Then he proved that all Heegaard splittings of a Seifert 3-manifold have distance  at most two; if a 3-manifold contains an essential torus, then all Heegaard splittings of it have distance at most two, where this result is also proved by Hartshorn \cite{ha} and Scharlemann \cite{sch01}. Combined with the Geometrization conjecture, if a 3-manifold admits a Heegaard splitting with Heegaard distance at least three,  then it is hyperbolic. So it seems that if we fully understand all of  distance two Heegaard splittings, then we can fully answer the question that how Heegaard splittings reflect 3-manifolds. So this question is reduced to
 \begin{ques}
 What dose a 3-manifold look like if it only admits distance at most two Heegaard splittings?
 \label{ques:1}
  \end{ques}
 Since the hyperbolic 3-manifolds are the most concerned, given a distance two Heegaard splitting, it is interesting to know whether the corresponding  manifold is hyperbolic or not.

By the definition of a distance two Heegaard splitting, there is an essential simple closed curve  and a pair of essential disks from different handlebodies so that this curve is disjoint from those two essential disks' boundaries. It seems that this Heegaard splitting is simple and hence  whether the manifold is hyperbolic or not should not be hard to answer. However, things for distance two Heegaard splittings are complicated because there are examples showing that a 3-manifold admitting a distance two Heegaard splitting could be Seifert, hyperbolic or contains an essential torus, see \cite{h01,AT,QZG,RT}.
 
Thompson \cite{AT} studied all distance two genus two Heegaard splittings and found that even for genus two  Heegaard splittings, those manifolds could be very complicated. Later, Rubinstein and Thompson \cite{RT} extended this result to genus at least three cases. But their results give no sufficient conditions to determine whether it is hyperbolic or not.

According to Geometrization conjecture, except small Seifert 3-manifolds and hyperbolic 3-manifolds, all of 3-manifolds contain essential tori. It is known that the Heegaard splittings of small Seifert 3-manifolds are well understood. So to answer the question \ref{ques:1}, the first step is to understand possible essential tori in a 3-manifold.

In \cite{QZG}, the authors studied the curve complex and introduced  the definition of a locally large geodesic. Then they constructed infinitely many arbitrary large distance Heegaard splitting. In the proof of Theorem 1.3 in  \cite{QZG}, they found that the locally large property of geodesics forces any geodesic realizing Heegaard distance to share some vertex $\gamma$  in common. So if the resulted manifold contains an essential torus $T^{2}$, then  $T^{2}$ intersects
this Heegaard surface in some essential simple closed curves, which are all isotopic to $\gamma$. Thus  $T^{2}$ intersects that Heegaard surface in fixed essential simple closed curves. Under this circumstance, it seems the corresponding 3-manifold is not hard to understand. For this reason, we introduce the definition of a locally large distance two Heegaard splitting.

A length two geodesic realizing Heegaard distance is $\mathcal{G}=\{\alpha, \gamma, \beta\}$, where $\alpha$ and $\beta$ bound essential disks from two sides of the Heegaard surface and $\gamma$ is disjoint from both $\alpha$ and $\beta$. As we know, there is a length two geodesic contains a non separating essential simple closed curve as its middle vertex realizing Heegaard distance. So we assume that $\gamma$ is represented by a non separating essential simple closed curve. Let $S$ be this Heegaard surface. We say the geodesic $\mathcal {G}$ is locally large if for the surface $S_{\gamma}=\overline{S-\gamma}$, $d_{\mathcal {S}_{\gamma}}(a, b)\geq 11$, for any pair of $a$ and $b$ disjoint from $\gamma$, where both $a$ and $b$ bound essential disks in different sides of $S$ respectively.  Moreover, by the definition of a locally large geodesic, we say a Heegaard splitting is locally large if there is a locally large geodesic realizing Heegaard distance.

The main result is
\begin{thm}
If a closed orientable manifold $M$ admits a locally large distance two Heegaard splitting $V\cup_{S}W$, then
$M$ is a hyperbolic manifolds or an amalgamation of a hyperbolic manifold and a small Seifert manifolds along an incompressible torus. Moreover there is only one essential torus in the non-hyperbolic case up to isotopy.
\label{thm1}
\end{thm}

A 3-manifold is almost hyperbolic if either  it is hyperbolic or it is  an amalgamation of a hyperbolic 3-manifold and a small non hyperbolic 3-manifold along an incompressible  torus.
Then the conclusion of Theorem \ref{thm1} says that
\begin{cor}
A 3-manifold admitting a locally large distance two Heegaard splitting is almost hyperbolic.
\end{cor}
The Geometrization Conjecture indicates that

(1) a Seifert 3-manifold does not admits a complete hyperbolic structure;

(2) a solvmanifold does not admits a complete hyperbolic structure;

(3) an amalgamation of a complete hyperbolic 3-manifold and a small Seifert 3-manifold along a torus  is not one of those eight geometries,
see \cite{scott}.

Thus combined with the result of Theorem \ref{thm1},

\begin{cor}
Neither a solvmanifold nor a Seifert 3-manifold admits a locally large distance two Heegaard splitting.
\end{cor}

According to Geometrization conjecture, the most important piece of those eight geometries is hyperbolic. Thus giving a sufficient condition for a hyperbolic 3-manifold is critical in studying Heegaard splittings. But the example in Section \ref{sec3}  shows that  the manifold $M$ in theorem \ref{thm1} could be a non hyperbolic manifold. So to give a sufficient condition for  a hyperbolic  3-manifold, we need to eliminate the possible essential torus in Theorem \ref{thm1}. For this purpose,  we introduce some definitions.

An essential simple  closed curve in Heegaard surface is a co-core for a handlebody if there is an essential disk in this handlebody so that its boundary intersects this curve in one point. The definition of a domain for an essential simple closed curve is in Section \ref{sec5}.
The proof of Theorem \ref{thm1} implies that under the locally large condition of $\gamma$, the non hyperbolic case happens  if
two copies of $\gamma$ bounds essential tori in both sides of $S$ and $\gamma$ is not co-core on either side of $S$. So,
\begin{thm}
 Suppose that a closed orientable manifold $M$ has a locally large distance two Heegaard splitting $V\cup_{S}W$. Let $\gamma$ be an essential simple closed curve disjoint from a pair of essential disks from different sides of $S$.  Then $M$ is hyperbolic if and only if either all domains of $\gamma$ have euler characteristic number less than -1 or $\gamma$ is co-core for one side of $S$.
\label{thm2}

\end{thm}

We introduce the definition of a geodesic of curve complex in Section \ref{sec2}, construct  a non hyperbolic 3-manifold in Section \ref{sec3} and prove theorem \ref{thm1} and \ref{thm2} in Section \ref{sec4} and \ref{sec5}.

{\bf Acknowledgements.} The authors would like to thank Tao Li for careful checking the earlier version of this manuscript, pointing out the genus two case of Theorem \ref{thm1} and many helpful conversations.

\section{Some needed Lemmas}
\label{sec2}
Let $S$ be a compact surface of genus at least 1, and $\mathcal{C}(S)$ be the curve complex of $S$.
We call a simple closed curve $c$ in $S$ is essential if $c$ bounds no disk in $S$ and is not parallel to $\partial S$. It is known that each vertex of $\mathcal {C}(S)$ is represented by the isotopy class of an essential simple closed curve in $S$. For simplicity, we do not distinguish the essential simple closed curve $c$ and its isotopy class $c$ without any further notation.

 Harvey \cite{h81} defined the curve complex
$\mathcal C(S)$ as follows: The vertices of $\mathcal C(S)$ are the
isotopy classes of essential simple closed curves on $S$, and $k+1$
distinct vertices $x_{0}, x_{1}, \ldots, x_{k}$ determine a
k-simplex of $\mathcal C(S)$ if and only if they are represented by
pairwise disjoint essential simple closed curves. For any two vertices $x$ and $y$
of $\mathcal C(S)$, the distance of $x$ and $y$, denoted by $d_{\mathcal{C}(S)}(x,
y)$, is defined to be the minimal number of 1-simplexes in a
simplicial path joining x to y. In other words,
$d_{\mathcal{C}(S)}(x,y)$ is the smallest integer $n \geq 0$ such that there is
a sequence of vertices $x_{0} = x, . . . , x_{n} =y$  such that
$x_{i-1}$ and $x_{i}$ are represented by two disjoint essential simple closed
curves on $S$  for each $1\leq i \leq n$. For any two sets of vertices in $\mathcal C(S)$, say $X$ and $Y$, $d_{\mathcal{C}(S)}(X,
Y)$ is defined to be $min\bigl\{d_{\mathcal{C}(S)}(x,
y) \| \ x\in X, \ y\in Y\bigr\}$. For the torus or once punctured torus case,  Masur and Minsky  \cite{mm99}
define $\mathcal{C}(S)$ as follows: The vertices of $\mathcal
C(S)$ are the isotopy classes of essential simple closed curves
on $S$, and $k+1$ distinct vertices $x_{0}, x_{1}, \ldots, x_{k}$
determine a k-simplex of $\mathcal C(S)$ if and only if $x_{i}$
and $x_{j}$ are represented by two simple closed curves $c_{i}$
and $c_{j}$ on $S$ such that $c_{i}$ intersects $c_{j}$ in just
one point for each $0\leq i\neq j\leq k$.
The following lemma is well known, see \cite{m, mm99, mm00}. \vskip 2mm

\begin{lem}
$\mathcal{C}(S)$ is connected, and the diameter of $\mathcal{C}(S)$ is infinite.
\label{lem2.1}
\end{lem}
A collection $\mathcal {G}=\{a_{0}, a_{1},...,a_{n}\}$  is a geodesic in $\mathcal {C}(S)$
if  $a_{i}\in \mathcal {C}^{0}(S)$ and \[d_{\mathcal {C}(S)}(a_{i},a_{j})=\mid i-j\mid,\]
for any $0\leq i,j\leq n$. And the length of $\mathcal {G}$ denoted by $\mathcal {L(G)}$ is defined to be $n$.
By Lemma \ref{lem2.1}, there is a shortest path in $\mathcal{C}^{1}(S)$ connecting any two vertices of $\mathcal {C}(S)$. Thus for any two distance $n$ vertices $\alpha$ and $ \beta$, a geodesic $\mathcal {G}$ connects $\alpha$ and $ \beta$ if $\mathcal {G}=\{a_{0}=\alpha,...,a_{n}=\beta\}$. Now for any two sub-simplicial complex $X, Y\subset \mathcal {C}(S)$,  a geodesic $\mathcal {G}$ realizing the distance of $X$ and $Y$ if $\mathcal {G}$ connects an element $\alpha\in X$ and an element $\beta \in Y$ so that $\mathcal {L(G)}=d_{\mathcal {C}(S)}(X,Y)$.\vskip3mm

Let $F$ be a compact surface of genus at least 1 with non-empty boundary. Similar to the definition of the curve complex $\mathcal {C}(F)$, we can define the arc and curve complex $\mathcal {AC}(F)$ as follows:

Each vertex of $\mathcal {AC}(F)$ is the isotopy
class of an essential simple closed curve or an essential properly embedded arc in $F$, and a set of vertices form a simplex of $\mathcal {AC}(F)$ if these vertices are represented by pairwise disjoint arcs or curves in $F$. For any two disjoint vertices, we place an edge between them. All the vertices and edges form 1-skeleton of $\mathcal {AC}(F)$, denoted by ${\mathcal {AC}}^{1}(F)$. And for each edge, we assign it length 1. Thus for any two vertices $\alpha$ and $\beta$ in ${\mathcal {AC}}^{1}(F)$, the distance $d_{\mathcal {AC}(F)}(\alpha, \beta)$ is defined to be the minimal length of paths in ${\mathcal {AC}}^{1}(F)$ connecting $\alpha$ and $\beta$. Similarly, we can define the geodesic in $\mathcal {AC}(F)$.\vskip 3mm

When $F$ is a subsurface of $S$, we call $F$ is essential in $S$ if the induced map of the inclusion from
$\pi_{1}(F)$ to $\pi_{1}(S)$ is injective. Furthermore, we call $F$ is a
proper essential subsurface of $S$ if $F$ is essential in $S$ and at least one boundary component of $F$ is essential in $S$. For more details, see \cite{mm00}.  \vskip 3mm

So if $F$ is an essential subsurface of $S$, there is some connection between the $\mathcal {AC}(F)$ and $\mathcal {C}(S)$. For any $\alpha\in \mathcal {C}^{0}(S)$, there is a representative essential simple closed curve $\alpha_{geo}$ such that the intersection number $i(\alpha_{geo}, \partial F)$ is minimal. Hence each component of $\alpha_{geo}\cap F$ is essential in $F$ or $S-F$. Now for $\alpha\in \mathcal {C}(S)$, let $\mathcal {\kappa}_{F}(\alpha)$ be isotopy classes of the essential components of $\alpha_{geo}\cap F$.

For any $\gamma\in \mathcal {C}(F)$, $\gamma'\in \mathcal {\sigma}_{F}(\beta)$ if and only if $\gamma'$ is the essential boundary component of a closed regular neighborhood of $\gamma\cup \partial {F}$.  Now let $\pi_{F}= {\mathcal {\sigma}_{F}}\circ {\mathcal {\kappa}_{F}}$. Then the map $\pi_{F}$ is the subsurface projection defined in \cite{mm00}.\vskip 2mm

We say $\alpha\in \mathcal {C}^{0}(S)$ cuts $F$ if $\pi_{F}(\alpha)\neq \emptyset$. If $\alpha$, $\beta\in \mathcal {C}^{0}(S)$ both cut $F$, we write $d_{\mathcal {C}(F)}(\alpha, \beta)= diam_{\mathcal {C}(F)}(\pi_{F}(\alpha),\pi_{F}(\beta))$. And if $d_{\mathcal {C}(S)}(\alpha, \beta)=1$, then $d_{\mathcal {AC}(F)}(\alpha, \beta)\leq 1$ and $d_{\mathcal {C}(F)}(\alpha, \beta)\leq 2$. \vskip 2mm

The following is immediately followed from the above observation.

\begin{lem}
 Let  $F$ and $S$ be as above, $\mathcal {G}=\{\alpha_{0},\ldots, \alpha_{k}\}$ be a geodesic of $\mathcal {C}(S)$ such that $\alpha_{j}$ cuts $F$ for each $0\leq i\leq k$. Then $d_{\mathcal {C}(F)}(\alpha_{0}, \alpha_{k})\leq 2k$.
 \label{lem:subsurface projection}
\end{lem}

For essential curves $\alpha, ~\beta$ in $S$, let $\mid \alpha\cap \beta\mid$ be the minimal geometric intersection number up to isotopy.
We call $\alpha$ and $\beta$ intersect efficiently if the number of $\alpha \cap \beta$ is equal to $\mid \alpha\cap \beta\mid$.

One tool for studying the intersection between essential simple closed curves and arcs in $S$ is bigon Criterion.

\begin{lem}\cite{FM}
Let surface $S$ be as above. Then for any two essential curves $\alpha, \beta$  in $S$,  $\alpha$ and $ \beta$ intersects efficiently if and only if  $\alpha\cup \beta\cup \partial S$ bounds no Bigon or half-bigon in $S$.
\label{lem:bigon criterion}
\end{lem}
Assume that $V$ is a non-trivial compression body, i.e., not the product I-bundle of a closed surface. Then there is an essential simple closed curve in $\partial _{+}V$ bounding an essential disk in $V$. Let $S$ be an essential subsurface in $\partial_{+}V$. We call
$S$ is a hole for $V$ if for any essential disk $D\subset V$, $\pi_{S}(\partial D)\neq \emptyset$. Furthermore, we call an essential subsurface
$S\subset \partial_{+}V$ is an incompressible hole for $V$ if $S$ is a hole for $V$ and is incompressible in $V$. Otherwise, $S$ is a compressible hole for $V$. Masur and Schleimer\cite{ms} studied the subsurface projection of an essential disk, and proved that:
\begin{lem}
Let $V$ be a non-trivial compression body and $S$ be a compressible hole for $V$. Then for an essential disk
$D$ in $V$, there are essential disks $D^{1}$ and $D^{2}$ satisfying:

$\cdot$  for $\partial D$, $\partial D^{1}$ and $\partial S$, they intersect efficiently;

$\cdot$  $\partial D^{2}\subset S$;

$\cdot$ there is one component of $\partial D\cap S$ is disjoint from an component of $\partial D^{1}\cap S$ and
$\partial D^{1}\cap \partial D^{2}=\emptyset$. Furthermore, $d_{\mathcal{AC}(S)}(\pi_{S}(\partial D), \partial D^{2})\leq 3$.
\label{lem:compressible hole}
\end{lem}
\begin{proof}
See the proof of Lemma 11.5 and Lemma 11.7 \cite{ms}.
\end{proof}

Let $\{x_{1}, x_{2},..., x_{n}\}$  be  a collection of some different points in  $S$. For the manifold $S\times S^{1}$, $\mathcal{SC}=\{x_{i}\times S^{1}, i=1,..,n\}$ is a collection of essential simple closed curves. A closed orientbale 3-manifold is Seifert if it is obtained by doing Dehn surgeries along $\mathcal {SC}$ as follows.  We remove a regular neighborhood of $x_{i}\times S^{1}$ and glue back a solid torus where the meridian curve coincides with some $\beta_{i}/\alpha_{i}$ slope. If $\beta_{i}/\alpha_{i}\neq 0$,
then  this fiber is called exceptional. As we know, if all of these exceptional fiber are removed, then $M-\cup N(\beta_{i}/ \alpha_{i})$ is $F$ $\times$  $S^{1}$. Somehow, $M$ is represented as follows: $M=\{S, \beta_{1}/\alpha_{1},..., \beta_{n}/\alpha_{n}\}$, where  $S$ is called the base surface. It is known that this representation is unique with some permutations in order.

In studying irreducible Heegaard splittings of a Seifert manifold $M$,  there are two standard ones named as vertical and horizontal Heegaard splittings. To be clear,  for a Seifert manifold $M=\{S, \beta_{1}/\alpha_{1},..., \beta_{n}/\alpha_{n}\}$ with projection $f:M\rightarrow S$, let $S= D\cup E\cup F$ be a cell decomposition where each component of $D$ or $F$  contains  at most one singular point in its interior and each component of $E$ is a square with one pair of opposite edges in $D$ and the other one in $F$, where both $D\cup E$ and $E\cup F$ are connected. Then the union of $H_{1}=f^{-1}(D)\cup E\times [0, \frac{1}{2}]$ is a handlebody  and $H_{2}$, the complement of $H_{1}$ in $M$,  is also a handlebody which is homeomorphic to $f^{-1}(F)\cup E\times[\frac{1}{2}, 1]$, where $S^{1}=[0,1]/\sim$. So $H_{1}\cup_{\partial H_{2}}H_{2}$ is a Heegaard splitting of $M$, called a vertical Heegaard splitting.  The construction of a vertical Heegaard splitting shows that for each Seifert manifold, it admits  a vertical Heegaard splitting. Knowing that fact, people wonder that whether there is another type Heegaard splitting for a Seifert manifold in general or not. However, there is no other type of Heegaard splitting for a Seifert manifold in general. In \cite{MS}, Moriah and Schultens proved that except some special cases, almost all of Seifert 3-manifolds admit only vertical Heegaard splittings.  They \cite{MS} also showed that there are horizontal Heegaard splittings for some Seifert 3-manifolds as follows. Taking a surface bundle $M_{1}=F\times I/(x, 0)\sim (\psi(x), 1)$, where $\chi(F)\leq 0$ with one boundary component and $\psi: F\times \{1\}\rightarrow F\times \{0\}$ is a periodic homeomorphism and fixes $\partial F$ point by point. Let $M$ be a Dehn filling of $M_{1}\cup D\times S^{1}$, where the longitude goes to $\partial F$. Then $\partial F\times \{0, \frac{1}{2}\}$ bounds an annulus $\mathcal {A}$ in $D\times S^{1}$ which cuts out an I-bundle of $\mathcal {A}$.  It is not hard to see that $F\times\{0,\frac{1}{2}\}\cup \mathcal {A}$ cuts $M$ into two handlebodies, where both of these two handlebodies are compact surfaces product I-bundles. So it gives a Heegaard splitting of $M$, called a horizontal Heegaard splitting. A result in \cite{MS} said that a Seifert manifold admits a horizontal Heegaard splitting if and only if its euler number is zero. Moreover, they \cite{MS} proved that all irreducible Heegaard splittings of a Seifert manifold is either vertical or horizontal.

 From the definition of a Seifert 3-manifold, if $M$ has a genus at least 1 base surface $S$ or $S^{2}$ but with at least 4 exceptional fibers, $M$ contains an essential torus $T^{2}$. Then for any strongly irreducible Heegaard splitting $H_{1}\cup_{\partial H_{1}} H_{2}$, there are two essential annuli $\mathcal {A}_{1}\subset H_{1}$ and $\mathcal {A}_{2}\subset H_{2}$ with $\partial \mathcal {A}_{1}\cap \partial \mathcal {A}_{2}=\emptyset$. If $M$ has $S^{2}$ as its base surface with at most three exceptional fibers, then

\begin{lem}
(1) for a vertical Heegaard splitting, there are essential disks $D_{1}$ and $D_{2}$ from two sides of Heegaard surface so that their boundaries intersects in at most two points;

(2) for a horizontal  Heegaard splitting,  there is an essential simple closed curve $C$ and  two essential annuli  $\mathcal{A}_{1}=C\times[0,\frac{1}{2}]$ in $H_{1}$ and $\mathcal {A}_{2}=C\times [\frac{1}{2},1]$ in $H_{2}$ so that $\partial {\mathcal {A}_{1}}\cap \partial{\mathcal {A}_{2}}$ contains at most one point.

\label{lem:Seifert1}
\end{lem}
\begin{proof}
The proof of second part is contained in the proof of Theorem 3.5 in \cite{h01}. So all we need to prove is the first part.
Since a weakly reducible Heegaard splitting satisfies the conclusion, we only consider all strongly irreducible vertical Heegaard splittings of it.
If this vertical Heegaard splitting has genus at least 3, Corollary 3.3 in \cite{h01} says that it has distance at most 1. Hence there are two essential disks satisfying the conclusion of Lemma \ref{lem:Seifert1}. If this vertical Heegaard splitting has genus 2, by the definition,  one handlebody $H_{1}$ is the union of two closed neighborhood of exceptional fibers and a rectangle $\times [0, \frac{1}{2}]$ and the other one $H_{2}$ is homeomorphic to  an I bundle of one-holed torus with a non trivial Dehn surgery. It is not hard to see that removing two exceptional fibers reduces $M$ into a torus $\times$  I with a non trivial Dehn surgery, where we do the Dehn surgery along an simple closed curve $C$ and the union of a longitude and $C$ bounds an embedded annulus. The rectangle $\times [0, \frac{1}{2}]$ is isotopic to the closed neighborhood of an properly embedded unknotted arc which connects these two boundaries. After removing a rectangle $\times [0, \frac{1}{2}]$, it is changed into the handlebody $H_{2}$.

Let $a$ be an properly embedded arc in this rectangle where it connects a pair of opposite edges. Then $a\times [0,\frac{1}{2}]$ bounds an essential disk $D_{1}$ in $H_{2}$. Let $b$ be an properly embedded essential  arc in once punctured torus which intersects the longitude empty. Then $b\times [\frac{1}{2}, 1]$ bounds an essential disk $D_{2}$ in $H_{2}$. It is not hard to see that $\partial D_{1}$ intersects $ \partial D_{2}$  in two points.
\end{proof}

Hempel \cite{h01} showed that for a vertical Heegaard splitting, its genus equal to the sum of the number of rectangles  and 1. So it means that only a Seifert 3-manifold with base surface $S^{2}$ with at most three exceptional fibers admits a genus 2 vertical Heegaard splitting. By the proof of Lemma \ref{lem:Seifert1}, for a strongly irreducible vertical Heegaard splitting of genus 2,

\begin{cor}
there are two essential disks $D_{1}$ and $D_{2}$ of two sides so that there are two non isotopy essential simple closed curve $C_{1}$ and
$C_{2}$ in Heegaard surface disjoint from both of them.
\label{cor:1}
\end{cor}

\begin{proof}
Let $D_{1}$ and $D_{2}$ be as in Lemma \ref{lem:Seifert1}. It is known that $H_{2}$ is an once punctured torus I-bundle with a non trivial Dehn surgery. Let $C_{1}$ be a longitude in upper boundary and $C_{2}$ be a longitude in lower boundary. Then these two curves satisfy the conclusion.
\end{proof}

\section{ A toroidal manifold with a distance 2 Heegaard splitting}
\label{sec3}
Let $M$ be a compression body with genus $2g-1$ , where $g\geq 2$, and $\partial_{-}M$ be a torus. Then there are
two non-separating spanning annuli $A_{1}$ and $A_{2}$, i.e., the boundary of an essential annulus lies in different components of $\partial M$, such that $\overline{M-A_{1}\cup A_{2}}$ are two handlebodies $V_{1}$ and $V_{2}$ with same genera, see Figure \ref{fig:spanning annulus}.

\begin{figure}[!htbp]
\begin{center}
\includegraphics[scale=2]{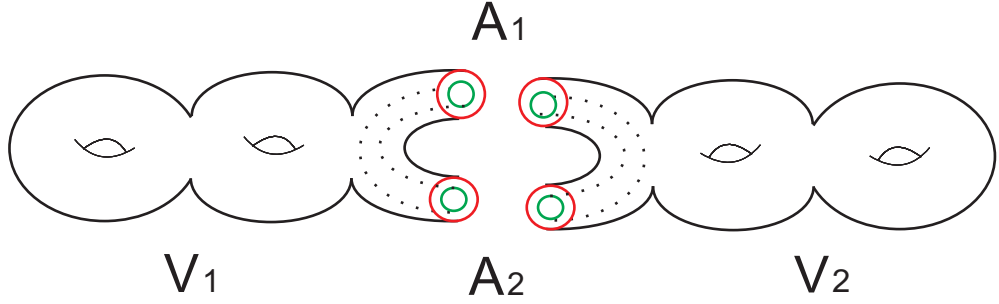}
\caption{Annuli in $M$}
\label{fig:spanning annulus}
\end{center}
\end{figure}

From Figure \ref{fig:spanning annulus}, $\partial V_{1}$ (resp. $\partial V_{2}$) consists of $S_{1}$ and an annulus $A_{1}^{1}$ (resp. $S_{2}$ and an annulus $A_{2}^{1}$). Since $S_{1}$ and $S_{2}$ are homeomorphic, there is a orientation reversing homeomorphism $f: S_{1}\rightarrow S_{2}$ such that
$f(\partial A_{1}^{1}\cap S_{1})=\partial A_{2}^{1}\cap S_{2}$.

Since $S_{1}$ is a genus $g-1\geq 1$ surface with two boundary,
the Projective Measured Lamination Space of $S_{1}$ \[\mathcal {PML}(S_{1})\cong S^{6g-9}\]
is not empty.
It is known that the isotopy class of the boundary $C\subset S_{1}$ of an essential disk in $V_{1}$ is an element of $\mathcal{PML}(S_{1})$. Then the collection of all essential simple closed curves bounding disks is a subset of $\mathcal{PML}(S_{1})$. It is known that the intersection function on $ML(S)$ defined a weak$^{*}$-topology on $ML(S)$, see \cite{penner}. Then there is a topology defined on $PML(S)$ induced by the projection $P:ML(S)\rightarrow PML(S)$. Under this topology, let $\mathcal {DS}_{1}\subset {PML}(S_{1})$ be the closure of all essential simple closed curves in $S_{1}$ which bound disks in $V_{1}$. So is $\mathcal {DS}_{2}$.
By the symmetry of these two handlebodies $V$ and $V_{2}$, there is an automorphism of $h:S_{1}\rightarrow S_{1}$ such that $h\circ f(\mathcal {DS}_{2})\subset \mathcal {DS}_{1}$.

\begin{fact}
$\mathcal {DS}_{1}$ is nowhere dense.
\label{fact:3.1}
\end{fact}
\begin{note}The proof is based on and contained in the proof of Theorem 1.2 \cite{masur}. For the integrity of this paper, we use the theory of Measured Lamination Space and rewrite it  here.
\end{note}

Before proving Fact \ref{fact:3.1},  we introduce a definition as follows. For any essential simple closed curve $\alpha\subset S_{1}$ bounding an essential disk in $V_{1}$, there is an disk system $\Gamma$ in $S_{1}$ such that

(1) one of its vertices is $\alpha$;

(2) all of its vertices are the isotopy classes of the boundaries of pairwise disjoint non-isotopic essential disks in $V_{1}$;

(3) it splits $S_{1}$ into a collection of  pairs of pants.

\begin{proof}
All  we need to prove is $\mathcal {DS}_{1}$ contains no open set in $\mathcal{PML}(S_{1})$.

Choosing an element $\alpha \in \mathcal {DS}_{1}$ represented by an essential non-separating simple closed curve in $S_{1}$, by above argument, there is
a disk system $\Gamma$ in $S_{1}$. For any element $\beta \in \mathcal {DS}_{1}$ represented by an essential simple closed curve in $S_{1}$,  by Lemma \ref{lem:bigon criterion}, we can isotope $\beta$ such that the intersection number $\mid \beta\cap \Gamma\mid$ is minimal. If $\beta$ intersects
$\Gamma$ nonempty, then there is an wave $w$ corresponding to the outermost disk component in the complement of $\Gamma$ in $S_{1}$. Since the boundary of $\partial S_{1}$ bounds no essential disk in $V_{1}$, the wave $w$ is contained in a pair of pants bounded by the boundaries of
essential disks. If $\beta$ intersects $\Gamma$ empty, then $\beta\in \Gamma$.

From Penner and Harer \cite{penner}, there is always a birecurrent maximal train track $\tau$ in $S_{1}$ such that it intersects all the wave like $w$ for the disk system $\Gamma$. Then there is a minimal measured lamination $\mathcal {L}$ carried by $\tau$  intersecting all the wave like $w$ such that the complement of it in $S_{1}$ is a disk or a one-holed disk with a finite points removed from its boundary, where the one holed disk contains one boundary of $S_{1}$. Then $\mathcal {L}$ is not in $\mathcal {DS}_{1}$ because it intersects each element  in $\mathcal {DS}_{1}$ non empty.

It is  known that the collection of essential simple closed curves in $S_{1}$ is dense in $\mathcal {PML}(S_{1})$. Then there is a sequence
$\{c_{1},..., c_{n},...\}$  converging to $\mathcal {L}$ in $\mathcal {PML}(S_{1})$, where $c_{i}$ is represented by an essential simple closed curve. Hence there is a number $N$ such that $c_{N+1}$ intersects all the waves like $w$ for the disk system $\Gamma$. So there is a neighborhood $U$ of $c_{N+1}$ in $\mathcal {PML}(S_{1})$ disjoint from  $\mathcal {DS}_{1}$ in $\mathcal {PML}(S_{1})$.

Now suppose that there is an open set $U^{'}\subset \mathcal {DS}_{1}$.
Then there is an automorphism $f:S_{1}\rightarrow S_{1}$, where $f(\mathcal {DS}_{1})=\mathcal {DS}_{1}$,
and a non separating essential curve $\alpha_{1} \in U^{'}$ bounding disk in $V_{1}$ such that
$f(\alpha_{1})=\alpha$ and $f(U^{'})\subset \mathcal {DS}_{1}$ is an open neighborhood of $\alpha$ in $\mathcal {PML}(S_{1})$.

For each essential simple closed curve $c\subset S_{1}$ which intersects $\alpha$ nonempty,
let $\tau_{\alpha}$ be the Dehn twist along $\alpha$ in $S_{1}$. It is known that $\tau^{n}_{\alpha}(c)$ is closed to $\alpha$ in $\mathcal {PML}(S_{1})$.
Then $\tau^{n}_{\alpha}(c_{N+1}))\subset f(U^{'})$ for $n$ large enough. Hence there is an open subset $U_{1}\subset U$ such that
$\tau^{n}_{\alpha}(U_{1})\subset f(U^{'})$. It means that $f^{-1}\circ \tau^{n}_{\alpha}(U_{1})\subset U^{'}$. Then $\tau^{-n}_{\alpha}\circ f(\mathcal {DS}_{1})\neq \mathcal {DS}_{1}$.
 But since  $\alpha$ bounds an essential disk in $V_{1}$, both of these two maps $\tau_{\alpha}$ and $\tau_{\alpha}^{-1}$ map
 $\mathcal{DS}_{1}$ into $\mathcal {DS}_{1}$. Hence $\tau_{\alpha}^{-n}\circ  f$ maps $\mathcal {DS}_{1}$ into
 $\mathcal {DS}_{1}$. A contradiction.

\end{proof}

Since the collection of those stable and unstable laminations of all pseudo anosov automorphisms in $S_{1}$ is dense in $\mathcal {PML}(S_{1})$,  there is a pseudo anosov map $g$ in $S_{1}$ such that the stable lamination are not in  $\mathcal {DS}_{1}$. By the proof of Theorem 2.7 \cite{h01}, if $n$ is large enough, then $d_{\mathcal {C}(S_{1})}(g^{n}(\alpha), h\circ f(\beta))\geq 11$ for any $\alpha$ and $\beta$ bounding essential disks in $V_{1}$ and $V_{2}$ respectively.

 For constructing  a non hyperbolic 3-manifold, we set $M_{1}$ be $V_{2}\cup_{g^{n}\circ h\circ f} V_{1}$ along $S_{2}$ and $S=\partial V_{1}$ in $M_{1}$. After pushing $S$  a little into the interior of $M$, $S$ splits $M_{1}$ into a handlebody $V$ and a compression body $W$, see Figure \ref{fig:Heegaardsurface2}.

\begin{figure}[!htbp]
\begin{center}
\includegraphics[scale=2]{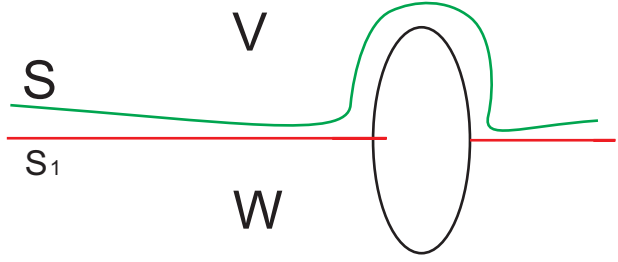}
\caption{Heegaard surface $S$}
\label{fig:Heegaardsurface2}
\end{center}
\end{figure}
Note: $S$ is colored in green and $S_{1}$ is colored in red, where $S$ is parallel to the union of $S_{1}$ and an annulus $A\subset \partial M_{1}$.

A Heegaard splitting is weakly reducible if there are a pair of essential disks from different sides of the Heegaard surface so that their boundaries intersect empty. Otherwise, the Heegaard splitting is strongly irreducible.
 \begin{fact}
 The Heegaard splitting $V\cup_{S}W$ is strongly irreducible.
 \label{fact3.2}
 \end{fact}
\begin{proof}
Suppose not.  Then the Heegaard splitting is weakly reducible. So there are a pair of essential disks $D\subset V$ and $E\subset W$ so that
$\partial D\cap \partial E=\emptyset$. From the construction of $M_{1}$,  $A$ is incompressible in $M_{1}$.
Let $S_{1,1}\subset S_{1}$ be $S_{1}-N(\partial S_{1})$, where $N(\partial S_{1})$ is a regular neighborhood of $\partial S_{1}$ in $S_{1}$.
After pushing  the closure of $A\cup N(\partial S_{1})$ a little into $M_{1}$ so that it is disjoint from $S_{1,1}$, $\overline{A\cup N(\partial S_{1})}$ is turned into an embedded annulus $A_{1,1}$. Then $S$ is isotopic to $S_{1,1}\cup A_{1,1}$.
It is not hard to see that every essential disk of $V$ (resp. $W$) has the property that its boundary cuts
$S_{1,1}$. It means that $S_{1,1}$ is a hole for both of $V$ and $W$. By the construction of $V\cup_{S}W$,  there is  an essential disk in $V$ (resp. $W$) with its boundary in $S_{1,1}$. Then $S_{1,1}$ is a compressible hole for both of $V$ and $W$.

By Lemma \ref{lem:compressible hole}, for the essential disk $D$, there is an essential disk $D_{1}\subset V$ such that

(1) $\partial D_{1}\subset S_{1,1}$;

(2) there is one component $a\subset \partial D\cap S_{1,1}$ such that $d_{\mathcal {C}(S_{1,1})}(\pi_{S_{1,1}}(a), \partial D_{1,1})\leq 3$;

Similarly for the essential disk $E$,  there is an essential disk $E_{1}\subset W$ such that

(1) $\partial E_{1}\subset S_{1,1}$;

(2) there is one component $b\subset \partial E \cap S_{1}$ such that $d_{\mathcal {C}(S_{1,1})}(\pi_{S_{1,1}}(b), \partial E_{1})\leq 3$;

Since $\partial D\cap \partial E=\emptyset$, by Lemma \ref{lem:subsurface projection}, $d_{\mathcal {C}(S_{1,1})}(\pi_{S_{1,1}}(a), \pi_{S_{1,1}}(b))\leq 2$. By triangle inequality, \[d_{\mathcal {C}(S_{1,1})}(\partial D_{1}, \partial E_{1})\leq 8 .\]
Since $S_{1,1}$ is an  essential subsurface of $S_{1}$, every essential simple closed curve in $S_{1,1}$ is an essential simple closed curve in
$S_{1}$. Then \[d_{\mathcal {C}(S_{1})}(\partial D_{1}, \partial E_{1})\leq 8 .\]
Since $S_{1}\subset \partial V_{1}$, it is not hard to see that $D_{1}$ is also an essential disk in $V_{1}$. So is the disk $E_{1}$. Then the inequality above implies that
\[d_{\mathcal {C}(S_{1})}(g^{n}(\alpha), h\circ f(\beta))\leq 8 ,\]
for some pair of  $\alpha$ and $\beta$ bounding essential disks in $V_{1}$ and $V_{2}$ respectively.

It contradicts the assumption of  $M_{1}$.
\end{proof}

 It is known that every Heegaard splitting of a boundary reducible 3-manifold is weakly reducible. Then the torus boundary $T^{2}_{1}$ of $M_{1}=V_{1}\cup_{f\circ g^{n}}V_{2}$ is incompressible.

Let $ST_{1}$ and $ST_{2}$ be two solid torus. Let $A_{1}^{2}\subset \partial ST_{1}$ be an incompressible annulus so that the core circle of $A_{1}^{2}$ intersects the meridian circle in at least two points up to isotopy. Similarly, choose an annulus $A_{2}^{2}$ in the boundary $ST_{2}$. After gluing $ST_{1}$ and $ST_{2}$ together along a homeomorphism between $A_{1}^{2}$ and $A_{2}^{2}$, the resulted manifold $M_{2}$ is a small Seifert 3-manifold with only one torus boundary $T^{2}_{2}$, where $T^{2}_{2}$ is incompressible.

Let $h_{1}:T^{2}_{1}\rightarrow T^{2}_{2}$ be a homeomorphism such that
$h_{1}(\partial S_{1})=\partial A_{1}^{2}$. Then $M^{*}=M_{1}\cup_{h_{1}}M_{2}$ is closed and $T^{2}_{2}$ is incompressible in $M^{*}$.

Let $S^{*}=S_{1}\cup A_{1}^{2}$.
Then $S^{*}$ splits $M^{*}$ into two 3-manifolds, denoted by $V^{*}$ and $W^{*}$ respectively.
In this case, $V^{*}$ is an amalgamation of $V_{1}$ and a solid torus $ST_{1}$ along the annulus $\partial V_{1}-S_{1}$. Then there are disjoint essential disks cutting $V^{*}$ into some 3-balls. So $V^{*}$ is a genus $g$ handlebody. Similarly, $W^{*}$ is a  genus $g$ handlebody too.
Hence $V^{*}\cup_{S^{*}}W^{*}$ is a Heegaard splitting of $M^{*}$.
\begin{fact}
$V^{*}\cup_{S^{*}}W^{*}$ is a distance 2 genus $g$ Heegaard splitting.
\label{fact:3.4}
\end{fact}
\begin{proof}
Since $\partial S_{1}$ is essential in $S^{*}$, every compression disk of $S_{1}$ in $M_{1}$ is an essential disk of $S^{*}$.
Since there are compression disks of $S_{1}$ in two sides, there are essential disks in $V^{*}$ and $W^{*}$ disjoint from $\partial A_{1}^{2}$ in $S^{*}$. Hence the Heegaard splitting $V^{*}\cup_{S^{*}}W^{*}$ has the distance less than or equal to two.

Suppose the Heegaard splitting $V^{*}\cup_{S^{*}}W^{*}$ has distance less or equal to one. Then there are two essential disks $D\subset V^{*}$ and $E\subset W^{*}$ so that $\partial D$ is disjoint from $\partial E$. It is not hard to see $S_{1}$ is a compressible hole for both $V^{*}$ and $W^{*}$.
By Lemma \ref{lem:compressible hole}, for the essential disk $D$, there is an essential disk $D_{1}\subset V^{*}$ such that

(1) $\partial D_{1}\subset S_{1}$;

(2) there is one component $a\subset \partial D\cap S_{1}$ such that $d_{\mathcal {C}(S_{1})}(\pi_{S_{1}}(a), \partial D_{1})\leq 3$;

Since $\partial S_{1}$ bounds an essential annulus in $V^{*}$, after some isotopy, $D_{1}$ is a compression disk for $S_{1}$ in $M_{1}$.

Similarly for the essential disk $E$,  there is an essential disk $E_{1}\subset W^{*}$ such that

(1) $\partial E_{1}\subset S_{1}$;

(2) there is one component $b\subset \partial E \cap S_{1}$ such that $d_{\mathcal {C}(S_{1})}(\pi_{S_{1}}(b), \partial E_{1})\leq 3$;

(3) $E_{1}$ is a compression disk for $S_{1}$ in $M_{1}$.

Since $\partial D\cap \partial E=\emptyset$, by Lemma \ref{lem:subsurface projection}, $d_{\mathcal {C}(S_{1})}(\pi_{S_{1}}(a), \pi_{S_{1}}(b))\leq 2$. By triangle inequality, \[d_{\mathcal {C}(S_{1})}(\partial D_{1}, \partial E_{1})\leq 8 .\]
It contradicts the assumption of $M_{1}$.






\end{proof}

By Fact \ref{fact:3.4}, $M^{*}$ admits a distance 2, genus $g$ Heegaard splitting.  Furthermore, it contains an essential torus.  Then there is a free abelian  subgroup $Z^{2}$ in its fundamental group. So $M^{*}$ is not hyperbolic.
\section{Proof of Theorem \ref{thm1}}
\label{sec4}

By the definition of Heegaard distance, for a distance 2, genus at least 2 Heegaard splitting $V^{*}\cup_{S^{*}}W^{*}$, there are three essential simple closed curves
$\{\alpha, \gamma^{*}, \beta\}$ so that $\alpha\cap \gamma^{*}=\emptyset, \gamma^{*}\cap \beta=\emptyset$ and $\alpha$ (resp. $\beta$) bounds an essential disk in $V$ (resp. $W$). Set \[\mathcal {G}=\{\alpha, \gamma^{*}, \beta\}.\] Then it is a geodesic in $\mathcal {C}(S)$ and realizes the Heegaard distance.

For the Heegaard splitting $V^{*}\cup_{S^{*}}W^{*}$, there maybe many geodesics in $\mathcal {C}(S^{*})$ realizing the Heegaard distance. Moreover it is unknown that whether all the geodesics have a common vertex or not. But for the non hyperbolic example $M^{*}=V^{*}\cup_{S^{*}} W^{*}$ in Section \ref{sec3}, there is an essential non-separating simple closed curve $\gamma^{*}$ such that for any pair of essential simple closed curves $\alpha$ and $\beta$ disjoint from $\gamma^{*}$ bounding essential disks in $V^{*}$ and $W^{*}$ respectively, $d_{\mathcal {C}(S_{\gamma^{*}})}(\alpha, \beta)\geq 11$.
\begin{fact}
Every geodesic realizing the distance of $V^{*}\cup_{S^{*}} W^{*}$ has $\gamma^{*}$ as one of its vertices.
\label{fact:4.1}
\end{fact}
\begin{proof}
Suppose not. Then there is one geodesic
\[\mathcal {G}_{1}=\{\alpha_{1}, \gamma_{1}, \beta_{1}\}\]
so that

(1) it realizes the Heegaard distance;

(2) $\gamma_{1}$ is not isotopic to $\gamma^{*}$.

Let $S_{\gamma^{*}}$ be the closure of the complement of $\gamma^{*}$ in $S^{*}$.
Since $\gamma^{*}$ bounds essential disks in neither $V^{*}$ nor $W^{*}$ and is non separating,  $S_{\gamma^{*}}$ is a compressible hole for both of these two disk complexes of $V^{*}$ and $W^{*}$.
By Lemma \ref{lem:compressible hole}, for $\alpha_{1}$ (resp. $\beta_{1}$), there is an essential disk $D$ ($E$) so that

(1) $\partial D$ (resp. $\partial E$) is disjoint from $\gamma^{*}$;

(2) there is an essential disk $D_{1}$ (resp. $E_{1}$) is disjoint from $ D$ (resp. $E$);

(3) there is one component of $a$ of $\alpha_{1}\cap S_{\gamma^{*}}$ (resp. $b$ of $\beta_{1}\cap S_{\gamma^{*}}$) disjoint from
one component of $\partial D_{1}\cap S_{\gamma^{*}}$ (resp. $\partial E_{1}\cap S_{\gamma^{*}}$).

Then by Lemma \ref{lem:subsurface projection}, \[d_{\mathcal {C}(S_{\gamma^{*}})}(\pi_{S_{\gamma^{*}}}(a),\partial D)\leq 3 ~\rm{and}~
d_{\mathcal {C}(S_{\gamma^{*}})}(\pi_{S_{\gamma^{*}}}(b), \partial E)\leq 3.\]
Since each component of $\gamma_{1}\cap S_{\gamma^{*}}$ is disjoint from $a$ and $b$ and  not isotopic to $\gamma^{*}$, by Lemma \ref{lem:subsurface projection},
\[d_{\mathcal {C}(S_{\gamma^{*}})}(\pi_{S_{\gamma^{*}}}(a),\pi_{S_{\gamma^{*}}}(b))\leq 4.\]
Then by triangle inequality,
$d_{\mathcal {C}(S_{\gamma^{*}})}(\partial D, \partial E)\leq 10$. It contradicts the assumption of $\gamma^{*}$.
\end{proof}

 For the closed orientable irreducible 3-manifold $M^{*}$, Geometrization Conjecture says that there are  finitely many essential tori so that after cutting $M^{*}$ along these tori, each piece is either hyperbolic, Seifert or Solvable. Thus to understand the geometry of $M^{*}$, the first thing is to check the possible embedded essential tori in it. It is known that for any possible essential torus $T^{2}$ in $M^{*}$, by Schultens' Lemma \cite{schultens},  they can be isotoped to a general position that $T^{2}\cap S^{*}$ consists of essential simple closed curves in both $S^{*}$ and $T^{2}$. After pushing the possible boundary parallel annulus to the other side, we assume that each component of $T^{2}\cap V^{*}$ (resp. $T^{2}\cap W^{*}$) is an essential annulus in $V^{*}$ (resp. $W^{*}$). On one side,  a boundary compression on an essential annulus produces an essential disk. So for each component of $\gamma\subset T^{2}\cap S^{*}$, there is a geodesic containing it as its one vertex, which realizes Heegaard distance. On the other side, by Fact \ref{fact:4.1},  each geodesic realizing distance of Heegaard splitting $V^{*}\cup_{S^{*}}W^{*}$ shares the same vertex $\gamma^{*}$. Hence each component of $T^{2}\cap S^{*}$ is isotopic to $\gamma^{*}$.

 After doing a boundary compression on one annulus component of
 $T^{2}\cap V^{*}$, there is an essential separating disk $D$ in $V^{*}$ so that

 (1) $D$ cuts out a solid torus $ST$ in $V^{*}$;

 (2) each component of $T^{2}\cap V^{*}$ lies in $ST$.

  The reason for  case (2) happening is that we  choose a boundary compression disk for $T^{2}\cap V^{*}$ so that its interior intersects them empty. Then after doing boundary compression along this disk,  the resulted disk $D$ is disjoint from all components of $T^{2}\cap V^{*}$. Since these two boundaries of this annulus are isotopic,  $D$ cuts out a solid torus from $V^{*}$. Then all components of $T^{*}\cap V^{*}$ are contained in this solid torus after isotopy.

  As all components of $T^{2}\cap V^{*}$  are pairwise disjoint,  all these components of $T^{2}\cap V^{*}$ are parallel, i.e., any two components of $T^{2}\cap V^{*}$ cuts out an I-bundle of annulus. So are $T^{2}\cap W^{*}$. Since the union of all these annuli is $T^{2}$,

 \begin{fact}
 $T^{2}$ intersects $V^{*}$ in only one essential annulus.
 \label{fact:4.2}
\end{fact}
\begin{proof}
Suppose not. Then there are at least two essential annulus in $V^{*}$. And there is an essential disk $D\subset V^{*}$ such that
$D$ cuts out a solid torus containing $T^{2}\cap V^{*}$. For $T^{2}\cap W^{*}$, there is also an essential disk $E\subset W^{*}$ such that $E$ cuts out a solid torus containing $T^{2}\cap W^{*}$, see Figure \ref{fig:nested annuli}.

\begin{figure}[!htbp]
\begin{center}
\includegraphics[scale=1.5]{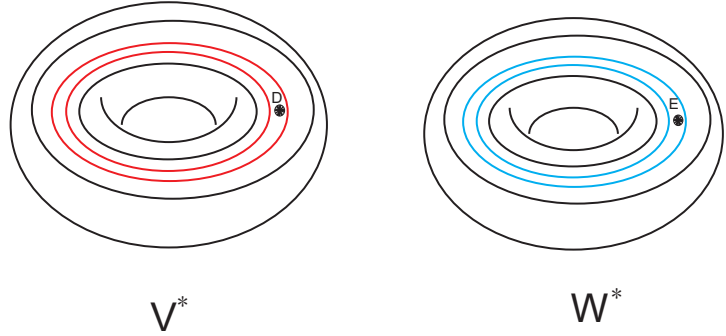}
\caption{Parallel Annuli}
\label{fig:nested annuli}
\end{center}
\end{figure}



Since the distance of Heegaard splitting $V^{*}\cup_{S^{*}}W^{*}$ is 2, $\partial D\cap \partial E\neq \emptyset$.
It means that the red circles coincide with the blue circles in Figure \ref{fig:nested annuli}. Then the essential
annulus bounded by the red circles in $V^{*}$ and the essential annulus bounded by the blue circles in $W^{*}$
are patched together in $T^{2}$. And the resulted manifold is a torus or a kleinian bottle. But $T^{2}$ contains no Kleinian bottle as its subset.
So  the resulted manifold is a torus which is the $T^{2}$, where it intersects $S^{*}$ in only
two simple closed curves. A contradiction.

\end{proof}

Moreover, the proof of Fact \ref{fact:4.2} indicates that

\begin{fact}
  if $M^{*}$ is toroidal, there is only one essential separating torus in $M^{*}$ up to isotopy.
  \label{fact:4.10}
 \end{fact}

We begin to prove Theorem \ref{thm1}, which is rewritten as follows.
\begin{thm}
For a manifold $M$ admitting a distance 2, genus at least 2 Heegaard splitting $V\cup_{S}W$, if there is an essential non-separating simple closed curve $\gamma$ in $S$ so that

(1) $\gamma$ bounds no essential disk in $V$ or $W$;

(2) there is a geodesic realizing Heegaard distance of $V\cup_{S}W$ with $\gamma$ as one of its vertices;

(3) for any pair of essential simple closed curves $\alpha$ and $\beta$ bounding disks in $V$ and $W$ respectively, if they are disjoint from
$\gamma$, then $d_{\mathcal {C}(S_{\gamma})}(\alpha, \beta)\geq 11$,

then $M$ is either a hyperbolic 3-manifold or an amalgamation of a hyperbolic 3-manifold and a small Seifert 3-manifold along an incompressible torus.
\label{thm:4.1}
\end{thm}
\begin{proof}

Since $M$ admits a distance 2 Heegaard splitting, by Haken's Lemma, $M$ is irreducible. It is known that every irreducible closed orientable 3-manifold $M$ either contains an essential torus or not. In the later case, by Geometrization conjecture,  $M$ is either a
small Seifert 3-manifold or a hyperbolic 3-manifold.

\begin{clm}
$M$ is not a small Seifert 3-manifold.
\label{clm:1}
\end{clm}
\begin{proof}
Suppose not. Then $M$ is a small Seifert 3-manifold. Hence it has $S^{2}$ as its base surface with at most three exceptional fibers.
If $M$ has only one or two exceptional fibers, then $M$ is a Lens space. But  all genus at least 2 Heegaard splitting of a Lens
space is stabilized, reducible, i.e., they all have distance 0. So $M$ contains three exceptional fibers.

Moriah and Schultens \cite{MS} proved that each irreducible Heegaard splitting of $M$ is either vertical or horizontal.  For the Heegaard splitting $V\cup_{S}W$, if it is  vertical,
then it has genus 2. By Corollary \ref{cor:1}, there are  two essential disks $D_{1}$ and $D_{2}$  from two sides of $S$ and two non isotopy  disjoint essential simple closed curves $C_{1}$ and $C_{2}$ so that both $C_{1}$ and $C_{2}$ are disjoint from $\partial D_{1}$ and $\partial D_{2}$. But under the condition that $d_{\mathcal {C}(S_{\gamma})}(\alpha, \beta)\geq 11$,  by the proof of Fact \ref{fact:4.1}, $C_{1}$ is isotopic to $C_{2}$. so it is impossible. Hence it is a horizontal Heegaard splitting.

Recall that for a horizontal Heegaard splitting, $M_{1}=F\times I/(x,0)\sim(\psi(x), 1)$, where $\partial F$ is connected and $\psi\mid \partial F\times I =Id$,  and $M=M_{1}\cup B^{2}\times S^{1}$. And $V=F\times [0, \frac{1}{2}]$ (resp. $W$ is homeomorphic to $F\times [\frac{1}{2}, 1]$).
By Lemma \ref{lem:Seifert1}, there is an essential simple closed curve $C\in F$ so that  $\mathcal {A}_{1}=C\times [0, \frac{1}{2}]$ and $\mathcal {A}_{2}=C\times [\frac{1}{2}, 1]$ so that $\partial \mathcal {A}_{1}$ intersects $\partial \mathcal {A}_{2}$ in at most one point.

It is not hard to see that there are a pair of essential disks of two sides of $S$ so that their boundary disjoint from $C\times\{\frac{1}{2}\}$.
By the proof of Fact \ref{fact:4.1}, $C\times \{\frac{1}{2}\}$ is isotopic to $\gamma$. Let $a$ be an arc in $F$ disjoint from $C$. Then
 there is an essential disk $D_{1}=a\times [0, \frac{1}{2}]$ (resp. $D_{2}=a \times [\frac{1}{2}, 1]$) disjoint from $C\times\{\frac{1}{2}\}$.  Thus $D_{1}\cap \mathcal {A}_{1}=\emptyset$ ( resp. $D_{2}\cap \mathcal {A}_{2}=\emptyset$ ). Hence
\begin{eqnarray*}
d_{\mathcal {C}(S_{\gamma})}(\partial D_{1}, \partial D_{2})&\leq& diam_{\mathcal {C}(S_{\gamma})}(\partial D_{1}, \partial \mathcal {A}_{1})+\\
                                                             &+&   diam_{\mathcal {C}(S_{\gamma})}(\partial \mathcal{A}_{1}, \partial \mathcal{A}_{2})+\\
                                                             &+&   diam_{\mathcal {C}(S_{\gamma})}(\partial D_{2}, \partial \mathcal {A}_{2})\\
                                                             &\leq &  1+2+1\\
                                                             &=& 4.
\end{eqnarray*}
It contradict the choice of $\gamma$.

\end{proof}

So $M$ is  hyperbolic or toroidal.

If $M$ is a hyperbolic manifold, then the proof ends. So we assume that $M$ contains an essential torus $T^{2}$.
By Fact \ref{fact:4.1}, \ref{fact:4.2} and \ref{fact:4.10},

(1) it contains only one essential torus $T^{2}$ up to isotopy, where
it is separating;

(2) each component of $T^{2}\cap S$ is isotopic to $\gamma$;

(3) $T^{2}\cap V$ (resp. $T^{2}\cap W$) splits $V$ (resp.$W$) into a solid torus and a handlebody.

Let $A$ be an annulus bounds by $T^{2}\cap S$ in $S$ and $S_{A}=\overline{S-A}=S_{\gamma}$.
Let $M_{1}$ be the amalgamation of these two solid tori along $A$.
It is not hard to see that $M_{1}$ is a small Seifert manifold with  a disk as its base surface.

Let $M_{2}=\overline{M-M_{1}}$. In the manifold $M_{2}$,  $\partial S_{A}$ consists of two isotopic essential simple closed curves in $\partial M_{2}=T^{2}$.
And $S_{A}$ cuts $M_{2}$ into two handlebodies.
Let $S_{2}$ be the union of $S_{A}$ and an annulus $A^{*}$ bounded by $\partial S_{A}$ in $\partial M_{2}$, see Figure \ref{fig:Heegaard surface 3}.
\begin{figure}[!htbp]
\begin{center}
\includegraphics[scale=1.5]{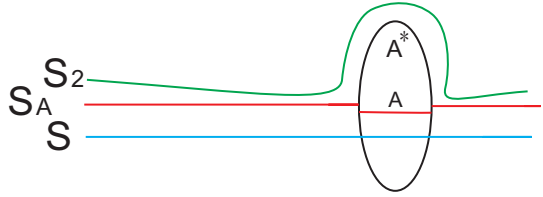}
\caption{Heeegaard surface $S_{2}$}
\label{fig:Heegaard surface 3}
\end{center}
\end{figure}

After pushing $S_{2}$ a little into the interior of ${M_{2}}$,
$S_{2}$ cuts $M_{2}$ into a handlebody and a compression body.
Then there is a Heegaard splitting $V_{2}\cup_{S_{2}}W_{2}$ for $M_{2}$. Similar to the proof of Fact \ref{fact3.2}, $S_{2}$ is a strongly irreducible Heegaard surface.

Remember that $S_{2}$ is also contained in $M$. So
\begin{fact}
$S_{2}$ share the essential subsurface $S_{A}$ with $S$ in common.
\label{fact:4.3}
\end{fact}
\begin{proof}
See Figure \ref{fig:Heegaard surface 3}.
\end{proof}
From Figure \ref{fig:Heegaard surface 3}, every essential disk in $V_{2}$ or $W_{2}$ with its boundary disjoint from $\partial S_{A}$ is a compression disks of $S_{A}$ in $V$ or $W$ respectively.

\begin{clm}
$M_{2}$ is irreducible, boundary irreducible, atoroidal and anannular.
\end{clm}
\begin{proof}
Since $M$ is irreducible and $T^{2}$ is incompressible, $M_{2}$ is irreducible and boundary irreducible.
By Fact \ref{fact:4.10} ,$M$ contains only one essential torus $T^{2}$ up to isotopy. Then $M_{2}$ is atoroidal.
Now suppose $M_{2}$ contains an essential annulus $A_{1}$. By Schultens' Lemma \cite{schultens},
$A_{1}\cap S_{2}$ are all essential simple closed curves in both $A_{1}$ and $S_{2}$.
After pushing all the boundary parallel annuli to the different side of $S_{2}$,
$A_{1}\cap V_{2}$ (resp. $A_{1}\cap W_{2}$) are essential annuli. We say at least one component $\gamma_{1}\subset A_{1}\cap S_{2}$ is not isotopic to $\gamma$.  For if not,  then there is an I-bundle of $\partial M_{2}=T^{2}$ containing $A_{1}$ after some isotopy,  which means that $A_{1}$ is inessential. Then there is an essential disk $D_{1}\subset V_{2}$ (resp. $E_{1}\subset W_{2})$) disjoint from $\gamma_{1}$.

Since $S_{2}$ cuts $M_{2}$ into a handlebody and a compression body, let $V_{2}$ be the handlebody.
From Figure \ref{fig:Heegaard surface 3}, $S_{2}$ is the union of $S_{A}$ and annulus $A^{*}$, where $V_{2}$ is a disk sum of a handlebody and I-bundle of the annulus $A^{*}$. Then the boundary of each essential disk in $V_{2}$ intersects $S_{A}$ nonempty.
So $S_{A}$ is a compressible hole. By a similar argument, $S_{A}$ is also a compressible hole for $W_{2}$.
Then by the proof of Fact \ref{fact:4.1}, there is a pair of essential disks $D\subset V_{2}$ for $D_{1}$ and $E\subset W_{2}$ for $E_{1}$ so that
$\partial D$ and $\partial E$ are both disjoint from $\partial S_{A}$ and
\[d_{\mathcal {C}(S_{A})}(\partial D, \partial E)\leq 10.\]
Remember that each essential disk in $V_{2}$ or $W_{2}$ disjoint from $\partial S_{A}$ is still an essential disk in $V$ or $W$ respectively  and
$S_{A}=S_{\gamma}$. Then it contradicts the choice of Heegaard splitting $V\cup_{S}W$.

\end{proof}

By Thurston's hyperbolic theorem of Haken manifolds, $M_{2}$ is hyperbolic.
\end{proof}
\begin{rmk}
The main result (Theorem 1.1) of Johnson, Minsky and Moriah's paper \cite{JMM} says that for a Heegaard splitting $V\cup_{S} W$, if there is an essential subsurface $F\subset S$ such that the distance of these two projections of  disk complexes $\mathcal {D}(V)$ and $\mathcal {D}(W)$ into $F$, denoted by $d_{F}(S)$, satisfies that $d_{F}(S)> 2g(S)+2$,  then up to an ambient isotopy, any Heegaard splitting of $M$ with genus less than or equal to $g(S)$ has the subsurface $F$ in common. For the Heegaard splitting in Theorem \ref{thm:4.1}, if  condition (3) is updated into $d_{\mathcal {S}_{1}}(\alpha, \beta)\geq \max\{2g(S)+3, 11\}$, then any Heegaard splitting $S^{'}$ of it with genus less than or equal to $ g(S)$ has $S_{1}$ in common up to an ambient isotopy. Since $\partial S_{1}$ bounds no disk in $M$, $S_{1}\subset S^{'}$ is essential.
By the calculation of the euler characteristic number,  $\partial S_{1}$ bounds an annulus $A$ in $S^{'}$.
The proof of Theorem \ref{thm1} implies that $A$ is parallel to an annulus in $S$. So the Heegaard $S$ is the unique minimal Heegaard surface up to isotopy.

\end{rmk}
\section{Proof of Theorem \ref{thm2}}
\label{sec5}

Let $M$, $V\cup_{S} W$ and $\gamma$ be the same as in Theorem \ref{thm:4.1}. On one side, since $\gamma$ is the middle vertex of a geodesic realizing Heegaard distance, there are two essential compression disks for $S$ disjoint from $\gamma$ from two sides of $S$. On the other side, as $\gamma$ is incompressible on both of these two sides of $S$, there is an essential subsurface $F\subset S$ containing $\gamma$ so that each essential simple closed curve in $F$ bounds no essential disk on either side of $S$. Then there is a maximal (defined later)  essential surface $F$ containing $\gamma$ so that there is no essential, i.e., incompressible and non peripheral, simple closed curve in $F$ so that it bounds an essential disk on either side of $S$.

It is possible that there are many  essential surfaces satisfying the property above. Thus we shall introduce some definitions for distinguishing all those surfaces. We call two subsurface $F_{1}$ and $F_{2}$ are same if $F_{1}$ is isotopic to $F_{2}$ in $S$. For a collection of different subsurfaces, we define a partial order as follows. For any two essential subsurface $F_{1}$ and $F_{2}$ of $S$, $F_{1}<F_{2}$ if $F_{1}$ can be isotopied into $F_{2}$ and $-\chi(F_{1})<-\chi(F_{2})$. Since there is a lowest bound for  all Euler characteristic numbers of those subsurface,  there is a maximal essential subsurface for any sequence of subsurfaces in order. For convenience,  for each one of these maximal essential subsurfaces,  we call it a domain of  $\gamma$.

Throughout the proof of Theorem \ref{thm:4.1}, the case that $M$ contains an essential torus means that (1) two copies of $\gamma$ bounds an essential annulus in both $V$ and $W$, namely, one domain of $\gamma$ is an once punctured torus in $S$; (2) $\gamma$ is not a co-core in either of these two sides of $S$.  So to eliminate the possible essential tori in $M$, it is sufficient to add some conditions related to these two cases (1) and (2).

We assemble the above argument as the following proposition.

\begin{pro}
Let $M$, $V\cup_{S}W$ and $\gamma$ be the same as in Theorem \ref{thm:4.1}. If either each domain of $\gamma$ has the Euler characteristic number less than -1  or $\gamma$ is  a co-core for one side of $S$, then $M$ is hyperbolic.
\label{pro:2}

\end{pro}

\begin{proof}
Suppose not. Then $M$ is not hyperbolic. Since $M$ admits a locally large distance two Heegaard splitting, by Theorem \ref{thm:4.1}, $M$ contains an essential annulus $T^{2}$.

The proof of Theorem \ref{thm:4.1} suggests that $T^{2}$ intersects $S$ in two copies of $\gamma$. It means that two copies of $\gamma$ bounds an essential annulus $A_{1}$ (resp. $A_{2}$) in $V$ (resp. $W$). Then there are two one hole tori of $S$ containing $\gamma$ in its interior from two
sides of $S$. Thus for either side of $S$,  there is one domain of $\gamma$ with Euler characteristic number equal to -1.

\begin{clm}
$\gamma$ is not a co-core for either side of $S$.
\end{clm}
\begin{proof}
Suppose not. Without loss of genericity, $\gamma$ is a co-core of the handlebody $V$. Then there is an essential disk $D$ so that $\partial D\cap \gamma$ in one point. Then $\partial N(\partial D\cup \gamma)$ bounds an essential disk $D_{1}$, which cuts $V$ into a solid torus $ST$ and a small genus handlebody. Since the annulus $A_{1}$ is essential in $V$ and $V$ is irreducible, by standard innermost disk surgery, $A\cap D_{1}=\emptyset$. Then $A_{1}\subset ST$.

As $\gamma$ is a co-core, the disk $D$ intersects  $A_{1}$  in one essential arc. Then there is a boundary compression disk $D_{0}\subset D$ for $A_{1}$ in $ST$ so that after doing a boundary compression along $D_{0}$,  $A_{1}$ is changed into a trivial disk in $V$. A contradiction.
\end{proof}

Thus these two conclusions contradict the assumption of $\gamma$.
\end{proof}

Moreover, the Proposition \ref{pro:2} can be updated into the following theorem, which is the Theorem \ref{thm2}.

\begin{thm}
Let $M$, $V\cup_{S}W$ and $\gamma$ be the same as in Proposition \ref{pro:2}. Then $M$ is hyperbolic if and only if  either each domain of $\gamma$ has the Euler characteristic number less than -1  or $\gamma$ is a co-core for one side of $S$.
\label{thm:5.3}
\end{thm}
\begin{proof}
For the forward direction. Suppose that (1) there are two domains $F_{1}$ and $F_{2}$ of $\gamma$, where both $F_{1}$ and $F_{2}$ are one hole disks and $\partial F_{1}$ (resp. $\partial F_{2}$) bounds an essential disk in $V$ (resp. $W$), and (2) $\gamma$ is not a co-core for both sides of $S$.
Then $\partial F_{1}$ (resp. $\partial F_{2})$) cuts out a solid torus in $V$ (resp. $W$) containing $\gamma$. Let $A$ be closed regular neighborhood of $\gamma$. Since $\gamma$ is not a co-core for either side of $S$, by the standard combinatorial techniques, $\partial A$ bounds two essential annuli $A_{1}$ and $A_{2}$  in both of $V$ and $W$ respectively, see Figure \ref{fig:5}.
 \begin{figure}[!htbp]
\begin{center}
\includegraphics[scale=1]{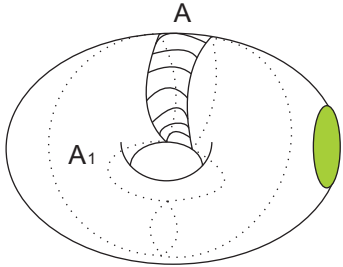}
\caption{The Essential Annulus $A_{1}$}
\label{fig:5}
\end{center}
\end{figure}

Then $A_{1}\cup A_{2}$ is a torus $T^{2}$ or a Kleinian bottle $K$. Since $A_{1}\cup A_{2}$ is separating in $M$,  it is a torus $T^{2}$.

\begin{clm}
$T^{2}$ is essential in $M$.
\label{clm:5.1}
\end{clm}
\begin{proof}
By Figure \ref{fig:5}, $A_{1}$ (resp. $A_{2}$) cuts out a solid torus $ST_{1}$ (resp. $ST_{2}$), where both of these two solid tori have the annulus $A$ as their common boundaries surface. Since $\gamma$,  the core curve of $A$,  is not a co-core of either of these two handlebodies $V$ and $W$, $M_{2}=ST_{1}\cup_{A} ST_{2}$ is a small Seifert space. Since $A$ is incompressible in both $ST_{1}$ and $ST_{2}$, $A$ is incompressible in $M_{2}$.
 For if $ T^{2}=\partial M_{2}$ is compressible in $M_{2}$, then the compression disk $D$ intersects $A$ nonempty up to isotopy. Otherwise, either $A_{1}$ or $A_{2}$ is compressible in $ST_{1}$ or $ST_{2}$ respectively. Then there is an outermost disk $D_{0}\subset D$ for $A$. Without loss of genericity,  we assume that $D_{0}\subset ST_{1}$ in $V$. It is not hard to see that $D_{0}$ is a boundary compression disk of $A_{1}$ in $V$. After dong boundary compression on $A_{1}$ along $D_{0}$,  $A_{1}$ is changed into a trivial disk in $V$, which is impossible.

Let $M_{1}=\overline{M-M_{2}}$. The proof of Fact \ref{fact3.2} suggests that $T^{2}$ is incompressible in $M_{1}$. So $T^{2}$ is incompressible in
$M$.
\end{proof}
So $M$ contains an essential torus $T^{2}$. It contradicts the assumption that $M$ is hyperbolic.

For the backward direction. The proof is contained in proof of Proposition \ref{pro:2}.
\end{proof}

\begin{rmk}
The conclusion of Theorem \ref{thm2} says that although these 3-manifolds which admit distance two Heegaard splittings are complicated,
we can still get some kind of  classification of them suggested by Geometrization Conjecture as we consider all those locally large distance two Heegaard splittings. 

\end{rmk}

 \vskip 3mm

Ruifeng Qiu\\
Department of Mathematics\\
 Shanghai Key Laboratory of PMMP\\
 East China Normal University\\
E-mail: rfqiu@math.ecnu.edu.cn\\

Yanqing Zou\\
 Department of Mathematics\\
  Dalian Nationalities University, Dalian Minzu University.
 \\E-mail: yanqing\_dut@163.com, yanqing@dlnu.edu.cn.\\
\end{document}